\documentclass[12pt,a4paper]{article}
\usepackage{amsfonts,amssymb}
\usepackage{latexsym}
\usepackage[usenames,dvipsnames]{color}
\usepackage{subfigure}
\usepackage{amsmath}
\usepackage{amsthm}
\usepackage{fullpage}
\usepackage{hyperref}

\newtheorem{thm}{Theorem}[section]

\newtheorem{cor}[thm]{Corollary}

\newcommand{\Z}{\mathbb{Z}}
\newcommand{\N}{\mathbb{N}}
\newcommand{\R}{\mathbb{R}}
\newcommand{\Q}{\mathbb{Q}}
\newcommand{\F}{\mathbb{F}}

\DeclareMathOperator{\dom}{dom}
\DeclareMathOperator{\ran}{range}
\newcommand{\rest}{\mathbin{\upharpoonright}}
\newcommand{\To}{\longrightarrow}

\newcommand{\st}{\; | \;}
\newcommand{\set}[2]{\left\{#1\st #2 \right\}}
\newcommand{\seq}[2]{\langle #1 \st #2 \rangle}

\renewcommand{\P}{\mathbb{P}}
\renewcommand{\a}{\textup{\textbf{a}}}
\renewcommand{\b}{\textup{\textbf{b}}}
\newcommand{\g}{\overline{g}}

\renewcommand{\d}{\textup{\textbf{d}}}

\begin{document}

\title{Infinite Latin Squares: Neighbor Balance and Orthogonality}

\author{Anthony~B.~Evans$^{1}$, Gage~N.~Martin$^{2}$, Kaethe Minden$^3$, and M.~A.~Ollis$^{3,}$\footnote{Corresponding author, email address: \texttt{matt@marlboro.edu.}}   \\
              \\
                {\it ${}^1$Department of Mathematics and Statistics, Wright State University,} \\    
              {\it Dayton, Ohio 45435, USA.} 
              \\
              \\
              {\it ${}^2$Department of Mathematics, Boston College,} \\
               {\it 140 Commonwealth Ave, Chestnut Hill,  } \\
              {\it  Massachusetts 02467-3804, USA.}
              \\
              \\
              {\it ${}^3$Marlboro College, P.O.~Box A, Marlboro,} \\    
              {\it Vermont 05344, USA.} }

\maketitle

\begin{abstract}
Regarding neighbor balance, we consider natural generalizations of $D$-complete Latin squares and Vatican squares from the finite to the infinite.   We show that if~$G$ is an infinite abelian group with~$|G|$-many square elements, then it is possible to permute the rows and columns of the Cayley table to create an infinite Vatican square. We also construct a Vatican square of any given infinite order that is not obtainable by permuting the rows and columns of a Cayley table.  Regarding orthogonality, we show that every infinite group~$G$ has a set of~$|G|$ mutually orthogonal orthomorphisms and hence there is a set of~$|G|$ mutually orthogonal Latin squares based on~$G$. We show that an infinite group~$G$ with~$|G|$-many square elements has a strong complete mapping; and, with some possible exceptions, infinite abelian groups have a strong complete mapping.

\vspace{3mm}
\end{abstract}

\section{Introduction}\label{sec:intro}

A finite Latin square is {\em row complete} or {\em Roman} if  any two distinct symbols appear in adjacent cells within rows once in each order.  If the transpose of a Latin square is row complete then the square is {\em column complete}; a square that is row complete and column complete is {\em complete}.  Finite row complete squares exist for all composite orders~\cite{Higham98} and finite complete squares are known to exist for all even orders~\cite{Gordon61} and many odd composite orders at which a nonabelian group exists; see, for example, \cite{Ollis14}.

Vatican and $D$-complete  squares strengthen this notion of completeness. 
A Latin square is is {\em row D-complete} if any two distinct symbols appear in cells that are distance~$d$ apart in rows at most once in each order for each~$d \leq D$. {\em Column D-completeness} is defined analogously and a square that is both row and column D-complete is {\em D-complete}.  The 1-completeness property is the same as completeness.

An $(n-1)$-complete square of order~$n$ is called {\em Vatican}; that is, Vatican squares have the pair-occurrence restriction at every possible distance.

Vatican squares are known to exist for all orders that are one less than a prime.  In addition to this, 2-complete squares are known to exist at orders~$2p$ where~$p$ is a prime congruent to 5, 7 or 19 modulo~24, orders~$2m$ where~$5 \leq m \leq 25$, and order~21 \cite{TuscanCRC,OllisTFSG}.

In this paper we extend these notions to the infinite and prove various existence results.  As in \cite{CW02}, we use Zermelo-Fraenkel set theory with the axiom of choice. In order to work with infinite sets, we use the set-theoretic machinery of ordinals and transfinite induction. An ordinal is an isomorphism type of well-ordered sets. The finite ordinals correspond to the natural numbers (or rather
the unique well-ordered sets with 0 or 1 or 2 etc. as elements), but there are also infinite ordinals.
The first infinite ordinal, denoted $\omega$, corresponds to the well-ordered set of natural numbers.
The ordinal corresponding to the well-order of the natural numbers with an added maximal element is
denoted $\omega+1$, and we can keep going after this to obtain $\omega+2, \omega+3$, and so on, reaching the limit $\omega+\omega$ (aka $\omega\cdot 2$), and beyond. 

As ordinals represent canonical well-orderings, they each support a notion of induction similar to the usual one on the natural numbers. Fixing an ordinal $\lambda$, transfinite induction up to $\lambda$ allows us to prove that a property $P$ holds for all ordinals below $\lambda$ by showing, from the hypothesis that $P$ holds for all ordinals below some $\alpha<\lambda$, that $P$ holds for $\alpha$ (in this formulation
the induction principle is analogous to strong induction on the natural numbers, and in fact reduces to it if we set $\lambda=\omega$).

It follows from the axiom of choice that for any set $X$, there is a bijection between $X$ and some ordinal. In the usual set-theoretic practice, the cardinality of $X$ is defined as the least such ordinal. 

We require a definition of an infinite Latin square that allows us to talk about spatial relationships.  This is accomplished by using a subset of an ordered field to index the rows and columns.  When that field is~$\Q$ or~$\R$, the infinite Latin squares we obtain are naturally embedded in~$\R^2$.

Let~$\F$ be an ordered field and let~$I \subseteq \F$.   For each~$d \in \F^+$ let~$I_{(d)} = \set{ i \in I }{  i+d \in I }$.  If $|I| = |\F|$ and for each~$d$ we have either $|I_{(d)}| = |\F|$ or~$|I_{(d)}| = 0$, then~$I$ is an {\em index set}.   For our purposes, we may assume that~$I_{(1)} \neq \emptyset$ without loss of generality.

Given an index set~$I$, a {\em Latin square} on~$I$ with symbol set~$X$ is a function $L: I \times I \rightarrow X$ such that for each~$i \in I$ the restriction of $L$ to $I \times \{i\}$ is a bijection with~$X$, as is the restriction to~$\{i\} \times I$.   In other words, each symbol appears once in each ``row" and once in each ``column."  

This definition is  compatible with the definition for Latin squares of arbitrary cardinality of Hilton and Wojciechowski~\cite{HW05}.  In the countable case with~$I = \N \subseteq \Q$ or $I = \Z \subseteq \Q$ we get the ``quarter-plane Latin squares" and ``full-plane Latin squares" respectively of Caulfield~\cite{Caulfield96}.

The definition of completeness for infinite squares is obtained by identifying the ideas of adjacency and being at distance~1.  This is perfectly natural when~$I \in \{ \N ,\Z\}$ and again matches the definition of Caulfield~\cite{Caulfield96}.   It does not seem to capture a property of particular combinatorial interest otherwise, but when we move to generalizing Vatican squares we get the very natural notion of pairs appearing once at all distances.  Indeed, the definition of an infinite Vatican square is arguably more natural than the finite version as it allows every pair to appear exactly once at every distance rather than merely at most once.

Formally, an infinite Latin square on an index set~$I$  is {\em row complete} or {\em Roman} if each pair of distinct symbols appears exactly once in each order at distance~1 in rows.  The square is {\em complete} if the corresponding property also holds in columns.   An infinite Latin square with indexing set~$I$ is {\em row $D$-complete} if each pair of symbols appear exactly once in each order at distance~$d$ in rows for each~$d$ such that~$I_{(d)} \neq \emptyset$ and $0 < d \leq D$. The square is~{\em $D$-complete} if the corresponding property also holds in columns. Further, the square is {\em Vatican} if for each~$d$ with~$I_{(d)} \neq \emptyset$ we have that each pair of distinct symbols appears at distance~$d$ exactly once in each order in rows and once in each order in columns.

Our first method for constructing squares uses Cayley tables of groups.   In the finite case all known constructions for complete squares--and hence $D$-complete and Vatican squares---use the notion of ``sequenceability" of a group and generalisations of it.  In the next section we show that similar notions are sufficient to construct infinite $D$-complete and Vatican  squares.  

Say that an infinite group~$G$ is~{\em squareful} if the set~$\{ g^2 : g \in G\}$ has the same cardinality as~$G$.  If~$G$ is an abelian squareful group and~$I$ is an index set with~$|I| = |G|$, then we can construct an infinite Vatican square on~$I$ using the Cayley table of~$G$. 

In Section~\ref{sec:notgp} we explore non-group-based methods.  We show that there is a Vatican square of each infinite order that cannot be produced by permuting the rows and columns of a Cayley table.  Whether a finite Vatican square with this property exists is an open question.  We also show that there is a Latin square of each infinite order such that no permutation of its rows and columns gives a Vatican (or even Roman) square.

As infinite sets can be bijective with proper subsets of themselves, we can define a variation on Vatican squares that only makes sense for infinite orders.
Say that an infinite Latin square on index set~$I$ is {\em semi-Vatican} if for each~$d$ with~$I_{(d)} \neq \emptyset$ we have that each pair of distinct symbols appears at distance~$d$ exactly once in rows and once in columns.  Although this does not have a finite analogue, all known constructions for finite Vatican squares of even order~$n$ have $n/2$ rows that together form a ``row semi-Vatican rectangle" and the remaining $n/2$ rows are the reverse of these ones.

All of the results for Vatican squares transfer to the semi-Vatican case in Section~\ref{sec:semivat} with little modification.  In addition to this, looking at the semi-Vatican case allows for an explicit construction of one in the case~$I = \R$ using only the tools of undergraduate Calculus.

Moving to orthogonality, two finite Latin squares on a symbol set~$X$ are orthogonal if for  each pair~$(x_1, x_2) \in X \times X$ there is exactly one position such that~$x_1$ is in that position in the first square and~$x_2$ is in that position in the second pair.  This definition carries over without modification to the infinite case (the countable version of which is given in \cite[p.~116]{DK15}).

In Section~\ref{sec:orth} we see that the methods from Section~\ref{sec:cayley} may be quickly adapted to produce sets of~$\kappa$ mutually orthogonal Latin squares of order~$\kappa$ for all infinite orders~$\kappa$ via Cayley tables of abelian squareful groups.  This is analogous to the finite construction of ``orthomorphisms" via ``R-sequencings".   We also construct orthomorphisms directly, finding $\kappa$ mutually orthogonal orthomorphisms for each group of infinite order~$\kappa$ and show that many infinite groups have strong complete mappings.

\section{Vatican squares from groups}\label{sec:cayley}

Let~$I$ be an index set in an ordered field~$\F$.  Let~$G$ be a group of order~$|I|$.  For a bijection~${\bf a}: I \rightarrow G$ define a function~${\bf a}_{(d)}: I_{(d)} \rightarrow G \setminus \{ e \}$ for each~$d \in \F^+$ with~$I_{(d)} \neq \emptyset$ by
$${\bf a}_{(d)}(i) = {\bf a}(i)^{-1}{\bf a}(i+d).$$ Such a function is called a $T_d$-sequencing for $\a$, if it is a bijection.
If there is a~$D$ such that for all~$d < D$ with~$I_{(d)} \neq \emptyset$ we have that each ${\bf a}_{(d)}$ is a bijection, then~${\bf a}$ is a {\em directed $T_D$-terrace for~$G$}.  If ${\bf a}_{(d)}$ is a bijection for all~$d$  with~$I_{(d)} \neq \emptyset$ then ${\bf a}$ is  a {\em directed $T_{\infty}$-terrace} for~$G$.

These definitions closely mimic the versions for finite groups~\cite{Anderson90}.   They can be used to produce Latin squares with neighbor balance properties in much the same way.  Theorem~\ref{th:terrace2square} generalizes Gordon's result~\cite{Gordon61} for finite complete squares and Anderson's~\cite{Anderson90} and Etzion, Golomb and Taylor's results~\cite{EGT89} for finite Vatican squares to the infinite.  

For any bijection~${\bf a}: I \rightarrow G$ define a square~$L({\bf a}) = (\ell_{ij})$ by $\ell_{ij} = a(i)^{-1}a(j)$.   As~${\bf a}$ is a bijection, each row and column contains each symbol exactly once and so~$L$ is a Latin square.  Call a Latin square created in this way {\em based on~$G$}, or simply {\em group-based}.

\begin{thm}\label{th:terrace2square}
Let~$G$ be a group of infinite order~$\kappa$.  If~$G$ has a directed $T_{D}$-terrace for an index set~$I$ then there is a $D$-complete Latin square of order~$|G|$ on~$I$.  Further, if~$G$ has a directed $T_{\infty}$-terrace then there is a Vatican square of order~$|G|$.
\end{thm}

\begin{proof}
Let~${\bf a}$ be a directed~$T_D$-terrace for~$G$ on~$I$ and consider $L({\bf a})$.

Take~$x$ and~$y$ to be distinct elements of~$G$.  As~${\bf a}_{(d)}$ is a bijection, there is a unique~$j$ with ${\bf a}(j)^{-1}{\bf a}(j+d) = x^{-1}y$ and a unique~$i$ with ${\bf a}(i)^{-1}{\bf a}(j)=x$.  We therefore have that~$x$ appears in row~$i$ and column~$j$ of~$L({\bf a})$ and that~$y$ appears in row~$i$ and column~$j+d$ of~$L({\bf a})$ and that~$x$ and~$y$ do not appear anywhere else with~$y$ exactly distance~$d$  to the right of~$x$.

There is also a unique~$i$ with $xy^{-1} = {\bf a}(i)^{-1}{\bf a}(i+d)$ and then a unique~$j$ with ${\bf a}(i)^{-1}{\bf a}(j)=x$.  This identifies a unique place where~$y$ appears at exactly distance~$d$ above~$x$ in the square.  Therefore $L({\bf a})$ is a $D$-complete square on~$I$.

If we replace~${\bf a}$ with a directed $T_{\infty}$-terrace in the above argument we see that~$L({\bf a})$ is a Vatican square on~$I$.
\end{proof}

The 1-complete case with~$I \in \{ \N, \Z \}$ of Theorem~\ref{th:terrace2square} is equivalent to results of Caulfield~\cite{Caulfield96}.

We wish to know which infinite groups have directed $T_D$- and~$T_{\infty}$-terraces.  

We use transfinite induction to build such terraces, and later to build the squares more generally. 
The construction will proceed by building better and better approximations to the object in transfinitely many steps, using each step to satisfy a requirement. Our approximations will be coherent and, at the end of the construction, we will end up with an object that satisfies all the requirements for the object we are trying to build.

We organize these constructions as follows. In each case we consider the partially ordered set $\P=\langle \P, \leq\rangle$, consisting of the approximations under consideration. Following established set-theoretic practice, we call elements of $\P$ conditions, and they are ordered so that $p\leq q$ if $p$ is a better approximation than $q$ (we say that $p$ is a stronger condition than $q$, or that it extends $q$).  Our inductive construction thus amounts to building a descending chain of conditions, at each step meeting a requirement. A major part of our proofs is showing that, given a requirement, any condition can be extended to meet it. This is usually phrased in terms of the set of conditions which satisfy the requirement being dense (in the sense of the order topology, i.e., for a set to be dense, any condition has an extension in this set). The key is is to judiciously pick out the dense sets in the poset which enable the object we are trying to build to clearly satisfy our desired properties.

\begin{thm}\label{th:T_infty}
Let~$G$ be an abelian squareful group. Then~$G$ has a directed $T_{\infty}$-terrace.
\end{thm}

\begin{proof}
Let $I$ be an index set in an ordered field $\F$, where $|I|=\kappa$ is the order of $G$. We build a directed $T_\infty$-terrace for $G$ by transfinite induction on $\kappa$. Consider the poset $\P$ consisting of partial directed $T_\infty$-terraces on $G$. These partial terraces are as in the definition of a directed $T_\infty$-terrace, except the functions $\a$ and its $T_d$-sequencings $\a_{(d)}$ for each $d \in \F^+$ with $I_{(d)} \neq \emptyset$ are only required to be injective partial functions from $I$ to $G$ with domains of cardinality less than $\kappa$. Here $\P$ should be partially ordered so that $\a \leq \b$ if and only if $\a$ extends $\b$ as a function, i.e.\ $\dom \b \subseteq \dom \a$ and $\a \rest \dom \b = \b$.

Below we identify the requirements the approximations have to meet and establish that the set of conditions satisfying each of them is dense.

\begin{enumerate}
	\item \label{item:DomainDense} For each $i \in I$, the set $D_i=\set{\d \in \P}{i \in \dom \d }$ is dense. 
	
	To see this, let $\a \in \P$ with domain $A$ and $i \in I \setminus A$. We need to find $\d \in D_i$ satisfying $\d \leq \a$. In order to find such a $\d$, first we must ensure that the value we assign to $i$ is not equal to anything in the range of $\a$, namely $\d(i) \neq \a(a)$ for each $a \in A$.
	Since the range of \(\a\) is smaller than \(\kappa\), this forbids fewer than \(\kappa\) many possible values for \(\d(i)\).
	
	Secondly, we must ensure the $T_d$-sequencings for $\d$ are injections. This amounts to ensuring that for each \(d\) and each $a \in A_{(d)}$, 
		$$\a(a)^{-1}\a(a+d) \neq \d(i)^{-1}\a(i+d)$$ and/or
		$$\a(a)^{-1}\a(a+d) \neq \a(i-d)^{-1}\d(i)$$ if $i-d$ and/or $i+d$ happen to be in $A$. Since there are strictly fewer than $\kappa$ many elements in the range of $\a$, this leaves fewer than $\kappa$ many elements of $G$ to avoid assigning $\d(i)$.
	In the case where $i-d$ and $i+d$ are both in $A$ we need to also make sure that $\d(i)^{-1}\a(i+d) \neq \a(i-d)^{-1}\d(i)$. As $G$ is abelian, this is the same as making sure that $$(\d(i))^2 \neq \a(i+d)\a(i-d).$$
	Again, since the range of $\a$ is small, there are fewer than $\kappa$ many forbidden
	values for $\d(i)^2$ and, since $G$ is squareful, this gives fewer than $\kappa$ new
	forbidden values for $\d(i)$.
	
	This means that altogether the set of values to rule out for $\d(i)$ has size less than $\kappa$, and we can just pick an element of \(G\) that has not been forbidden and assign it to \(\d(i)\).
	Then \(\d\) is a partial \(T_\infty\)-terrace with \(i\) in its domain. 
	
	\item For each $g \in G$, the set $D_g=\set{\d \in \P}{g \in \ran \d }$ is dense. 
	
	Again, as in the above case, the idea should be that we only have to avoid fewer than $\kappa$ many cases, but we have room in $I$ for that.
	
	Let $\a$ be a condition with $g \in G \setminus \ran \a$. We need to find $\d \in D_g$ satisfying $\d \leq \a$. This amounts to finding $\g$ so that we can let $\d(\g)=g$.
	This $\g$ of course cannot be in $A=\dom \a$. We also choose such that $\g\pm d$ is not
	in $A$ for any $d$ such that $A_{(d)}$ is nonempty, and such that $\g \neq \frac{a+a'}{2}$ 
	for any pair $a,a'\in A$. This avoids any issues in the partial $T_d$-sequencings and again forbids only fewer than $\kappa$
	many values for $\g$, so we can make a suitable choice. 
		
	\item For each $g \in G$ and each $d \in \F^+$ with $I_{(d)}\neq \emptyset$, the set $D^d_g=\set{\d \in \P}{g \in \ran \d_{(d)}}$ is dense. 
	
	To see this, fix $d \in \F^+$ such that $I_{(d)}\neq \emptyset$ and let $g \in G$. Let $\a \in \P$, and suppose that $g \notin \ran\a_{(d)}$. We want to see that it is possible to extend $\a$ to a condition $\d \in D^d_g$ such that $g=\d(\overline g)^{-1}\d(\overline g+d)$ for some $\overline g \in I$. This amounts to finding a suitable $\overline g$. First we need $\overline g$ to be so that $\overline g \notin A_{(d)}$ where $A = \dom \a$. Then we need to ensure that $\d(\overline g), \d(\overline g+d) \notin \ran \a$, and also obviously that $g=\d(\overline g)^{-1}\d(\overline g+d)$. Also make sure that neither $\g$ nor $\g+d$ fall halfway between elements of $A$.

	It must also be the case that for any $a \in A$, we have that 
		$$\a(a)^{-1}\d(\overline g) \notin \ran \a_{(\overline g-a)}, \ \  \d(\overline{g})^{-1}\a(a) \notin \ran \a_{(a-\overline g)},$$ 
		$$\a(a)^{-1}\d(\overline g+d)  \notin \ran \a_{(\overline g+d-a)}, \ \ \d(\overline g+d)^{-1}\a(a) \notin \ran \a_{(a-\overline g-d)}.$$

	Since we have only eliminated less than $\kappa$ many options, as we are restricted by $A$ and its image under $\a$, we have plenty of room to choose a $\overline g$ as desired.

We should note that we will not inadvertently disrupt another sequencing. In particular, if we have that both $\overline g+d', \overline g+d+d'\in A$ for some $d'\in I^+$ and $$\d(\overline g)^{-1}\a(\overline g+d') = \d(\overline g+d)^{-1}\a(\overline g+d+d'),$$ then, as $G$ is abelian, we must satisfy $$g= \d(\overline g)^{-1}\d(\overline g+d) = \a(\overline g+d')^{-1}\a(\overline g+d+d').$$ But this contradicts the requirement that $g \notin \ran \a_{(d)}$. Dually, whenever there is some $d'\in I^+$ such that $\overline g-d', \overline g-d-d' \in A$, and $$\a(\overline g-d')^{-1}\d(\overline g) = \a(\overline g+d-d')^{-1}\d(\overline g+d),$$ and this again contradicts $g \notin \ran \a_{(d)}$.
\end{enumerate}

Let $$\mathcal D = \set{D_i}{i\in I} \cup \set{D_g}{g\in G} \cup \set{D^d_g}{d \in \F^+\text{ with } I_{(d)}\neq \emptyset, g \in G},$$ and note that $|\mathcal D|=\kappa$, so we may enumerate all of the dense sets as $\mathcal D = \seq{\mathcal D_\alpha}{\alpha<\kappa}$.

We will define a descending sequence of conditions $\seq{\b_\alpha}{\alpha<\kappa}$ by transfinite induction, ensuring that $\b_\alpha\in \mathcal D_\alpha$ and $|\b_\alpha|<|\alpha+\omega|$, so that the sequence isn't growing at too fast a rate, at each step. (We saw it is possible to meet the dense sets without growing the conditions too fast while showing that each set in $\mathcal D$ is dense.) Assume that we have built an initial segment of this sequence $\seq{\b_\alpha}{\alpha<\lambda}$, and we wish to construct the condition $\b_\lambda$ in the next step. First notice that $\b_\lambda'=\bigcup_{\alpha<\lambda}\b_\alpha$ is itself a condition. This is because the only requirements are that $\b_\lambda'$ and $(\b_\lambda')_{(d)}$ be injective functions, and this
will be true if it was true for every earlier $\b_\alpha$. Furthermore, the domain of $\b_\lambda'$ is just the union of the domains of the earlier $\b_\alpha$ and each $\b_\alpha$ has size at most $|\alpha+\omega|$, 
so $\b_\lambda'$ itself is smaller than $\kappa$. We can now let $\b_\lambda$ be any extension of $\b_\lambda'$ in $\mathcal{D}_\lambda$; such an extension exists since we showed that
$\mathcal{D}_\lambda$ is dense.

By construction, $\bigcup_{\alpha<\kappa}\b_\alpha$ defines a directed $T_\infty$-terrace $\b: I \To G$ as desired:

\begin{enumerate}
	\item \emph{$\b$ is a bijection:} This is ensured by meeting, for each $i \in I$, the dense sets $D_i$ for injectivity and for meeting $D_g$ for each $g \in G$ for surjectivity.
	\item \emph{For each $d \in I^+$, $\b_{(d)}$ is a bijection:} The fact that the sequencing is injective is ensured by item \ref{item:DomainDense} as well, since at some point we will add both $i \in I$ and $i+d$ to the domain of the partial terrace we are constructing. The dense sets $D^d_g$ for each $g \in G$ guarantee surjectivity. \qedhere
\end{enumerate}
\end{proof}

\begin{cor}\label{cor:vatsquares}
For every index set~$I$ there is a Vatican square on~$I$.  In particular, there is a Vatican square of every infinite order.
\end{cor}

\begin{proof}
For every infinite order~$\kappa$ there is an abelian squareful group~$G$ of order~$\kappa$. Use~$G$ in Theorems~\ref{th:terrace2square} and~\ref{th:T_infty} to produce the required Vatican square.
\end{proof}

The constraints on~$G$ in Theorem~\ref{th:T_infty} were due to the method of proof rather than fundamental impediments.  Which other  infinite groups admit directed $T_D$- and~$T_{\infty}$-terraces?  Are there any that do not?  Vanden Eynden shows that all countably infinite groups have a directed 1-terrace on~$\N$ \cite{VE78} and the proof is easily adapted to apply to index set~$\Z$.  

The proof method of Theorem~\ref{th:T_infty} can be applied to an additional family of groups:

\begin{thm}\label{th:allinv}
If every non-identity element of an infinite abelian group $G$ is an involution, then $G$ has a directed~$T_{\infty}$-terrace for any index set of size~$|G|$. 
\end{thm}

\begin{proof}
Such a group~$G$ is not squareful, but otherwise meets the constraints of Theorem~\ref{th:T_infty}.  Apply the same proof (with the same notation).  There is only one point at which we use squarefulness: given a value of~$i$ we need to choose a value for~${\bf d}(i)$ such that
$$(\d(i))^2 \neq \a(i+d)\a(i-d)$$
for all~$d$ for which $\a(i+d)$ and $\a(i-d)$ are both defined.  As $(\d(i))^2$ is necessarily the identity and $\a(i+d) \neq \a(i-d)$, this inequality unavoidably holds in~$G$.
\end{proof}

\section{Squares not based on groups}\label{sec:notgp}

The results of the previous section raise the question about what is and is not possible for infinite squares more generally.   Perhaps {\em all} countably infinite squares may be made complete, or even Vatican, with a suitable permutation of their rows and columns?   In a similar vein, it is known that all infinite Steiner triple systems are resolvable~\cite{DHW14}, an uncommon property among finite systems. However, Theorem~\ref{th:notrcls} eliminates this possibility, showing that for every index set (and hence every infinite order) there is a square that cannot be made row-complete via permuting columns.

In the other direction, a question asked (and answered positively) about finite squares was whether there exist row-complete Latin squares that are not based on groups, see \cite{CE91, DK15, Owens76}.  We answer the infinite version of this question, also positively, in Theorem~\ref{th:infvat}.  Indeed, this result gives a Vatican square that is not group-based for every index set.   All known finite Vatican squares are based on groups.

\begin{thm}\label{th:notrcls}
For every index set~$I$ there is a Latin square on~$I$ that cannot be made row-complete by permuting columns.
\end{thm}

\begin{proof}
We build a Latin square on an ordered field $\F_\kappa$ of size $\kappa=|I|$ by approximations to it, by transfinite induction on $\kappa$. The same procedure would give one on $I$.

Consider the poset $\P$ consisting of injective partial functions from $\F_\kappa\times\F_\kappa$ to $\kappa$ whose domains have size strictly less than $\kappa$. We think of these as populating a $\kappa$-by-$\kappa$ grid with ordinals less than $\kappa$---each condition fills in some portion, or region, of the full grid. These regions must satisfy that each ordinal appears at most once in each row/column (they are Latin). Moreover we require that the Latin regions in $\P$ are immune: a region is \textit{immune} if every Latin region  obtainable by a permutation of its columns has a repeated occurrence of an ordered pair at horizontal distance~1. Again, partially order the immune Latin regions by $l\leq m$ if $l$ extends $m$ as functions.

We will start our transfinite induction with the following finite immune Latin region, a condition which we name $l_{-1}$:
$$
\begin{array}{ccc}
2  & 0 & 1 \\ 
1 & 2 &  0  \\
 0  & 1 & 2 
\end{array}
$$
We again list the requirements that the approximations will have to meet and show that densely many conditions satisfy these requirements.
\begin{enumerate}
	\item For each $\alpha, \beta<\kappa$, the set $D_{\alpha,\beta}=\set{l \in \P}{\alpha \text{ appears in row } \beta \text{ of } l}$ is dense.
	
	In other words, we need to see how to add a desired number $\alpha$ to row $\beta$ of a condition $l$, if it isn't already there. Place $\alpha$ into the $\beta$-th row, again without violating the Latin constraint (this might require us to introduce a new column).
	
	We have produced a larger Latin rectangle $l'$ which contains $l$ and has $\alpha$ in the $\beta$-th row but may not be immune, if we introduced a new column. We now perform the following immunization procedure to extend $l'$ to an immune Latin region in $D_{\alpha,\beta}$.
\begin{description}
		\item[Immunization:] For each combination of three different nonempty columns from our new region $l'$, pick the least number $\gamma$ larger than the largest number used in $l'$, and add the following three columns in entirely new rows above the three different columns: 
				$$\begin{array}{ccccc}
			\gamma+2  && \gamma   &&  \gamma+1 \\ 
			\gamma+1  && \gamma+2 &&  \gamma \\
			\gamma  && \gamma+1 && \gamma+2
		\end{array}$$ in whatever order you would like. Keep doing this with every possible combination of three different columns until all of the possible combinations of 3 different columns have been exhausted. Then fill in the gaps with whatever you would like without contravening the Latin constraints. 
		\end{description}
		To illustrate the technique, let's say that we would like to add the number $3$ to row $0$ in our initial immune Latin rectangle $l_{-1}$. Following the described algorithm, we first produce
		the following region $l_{-1}'$ 
		$$\begin{array}{cccc}
			2  & 0   &  1 & .  \\ 
			1 & 2 &  0  & . \\
			 0  & 1 & 2 & \color{red} 3
		\end{array}$$
		We added 3 to the first row of a new column. In this example, there are 4 columns available and thus 4 possible combinations of 3 columns, so we end up adding 12 new rows to the top of our condition.		
		$$\begin{array}{cccc}
						. &\color{blue}15& \color{blue}13&\color{blue}14\\
						. &\color{blue}14&\color{blue}15&\color{blue}13\\
						. &\color{blue}13&\color{blue}14&\color{blue}15\\
						\color{cyan}12&.&\color{cyan}10&\color{cyan}11\\
						\color{cyan}11&.&\color{cyan}12&\color{cyan}10\\
						\color{cyan}10&.&\color{cyan}11&\color{cyan}12\\
						\color{blue}9&\color{blue}7&.&\color{blue}8\\
						\color{blue}8&\color{blue}9&.&\color{blue}7\\
						\color{blue}7&\color{blue}8&.&\color{blue}9\\
						\color{cyan}6&\color{cyan}4&\color{cyan}5&.\\
						\color{cyan}5&\color{cyan}6&\color{cyan}4&.\\
						\color{cyan}4&\color{cyan}5&\color{cyan}6&.\\
						2  & 0   &  1 &. \\ 
						1 & 2 &  0 & . \\
						0  & 1 & 2 & {\textcolor{red} 3}
		\end{array}$$
	
		\item For each $\alpha, \beta<\kappa$, the set $D^{\alpha, \beta}=\set{l\in \P}{\alpha \text{ appears in column } \beta \text{ of } l}$ is dense. 
		
		Here we need to add a specified number $\alpha$ to the column number $\beta$ of a condition $l$, if it isn't already there. The strategy is the same as before: place $\alpha$ in column $\beta$ as required so that the region $l'$ produced is still Latin. Then run through the immunization procedure as outlined above if $l$ doesn't have any entries in column $\alpha$ to begin with, to produce a new condition extending $l'$ that is in $D^{\alpha, \beta}$.
		
		\item For each $\alpha, \beta < \kappa$, the set $E_{\alpha, \beta}=\set{l \in \P}{\text{$l$ has a symbol in row $\alpha$, column $\beta$}}$ is dense.
		
		Meeting this dense set ensures that we fill up all the holes we might have left when we met other dense sets. Given an immune Latin region $l$, in order to add something to position $(\alpha, \beta)$ simply place a symbol that is not in $l$ in that position to produce $l'$. If $l$ did not have anything in column $\beta$ yet, run through the immunization procedure outlined above on $l'$ to make it immune.
\end{enumerate}

We enumerate the listed dense sets into a sequence $\set{D_\alpha}{\alpha<\kappa}$ and build a descending chain of conditions $\seq{l_\alpha}{\alpha<\kappa}$ by transfinite induction, ensuring that $l_\alpha\in D_\alpha$ at each step and that $|l_\alpha|<|\alpha+\omega|$ so that the sequence isn't growing at too fast a rate. Note that if a condition $m$ has infinite size $\mu$, the immunization process will add at most $|\mu^3 \cdot 9|$ many new entries to a condition, which has size $\mu$. So the growth constraint on the sequence is reasonable given how we chose to meet dense sets in the above.

Proceed the same way as in the proof of Theorem~\ref{th:T_infty}. Given an initial segment of this sequence $\seq{l_\alpha}{\alpha<\lambda}$, notice that $l_\lambda'=\bigcup_{\alpha<\lambda}l_\alpha$ is a condition, since the growth rate of the sequence is not too high. We then let $l_\lambda$ be any extension of $l_\lambda'$ in $D_\lambda$.

We have produced $L=\cup \set{l_\alpha}{\alpha<\kappa}$, which defines a Latin square with the desired properties. Indeed, by meeting the dense sets we have ensured that the square is total and that every column and row contain all ordinals below $\kappa$. Moreover, if $L$ failed to be Latin, this failure would show up in some approximation $l_\alpha$, but these are all Latin, so $L$ must be as well. Finally, it isn't possible to permute columns in $L$ to obtain a row-complete square. To see this, suppose otherwise, that we could produce a row-complete Latin square $L^*$ by permuting columns of $L$. Then take the first three columns of $L^*$. These correspond to some columns, say $\alpha$, $\beta$, and $\gamma$, of $L$. There has to be a condition, say $l_\delta$, in which all three of these columns are nonempty for the first time. In that stage, we made sure that $l_\delta$ was immune; meaning every region obtained by a permutation of its columns has a repeated occurrence of an ordered pair at horizontal distance 1. In the way we performed the immunization procedure, we made sure that if you were to permute columns so that columns $\alpha$, $\beta$, and $\gamma$ were all ``next" to each other (according to the ordered field $\F_\kappa$), there would be a $3\times3$ Latin square somewhere, which can't be row-complete. This contradicts $L'$ being row-complete. So $L$ is a Latin square of size $\kappa$ which cannot be made row-complete by permuting columns, as desired.
\end{proof}

Prior to giving Theorem~\ref{th:infvat} we need a result that lets us be sure that a square is not group-based.  The {\em quadrangle criterion} states that in a square based on a group if the three equations
$$a_{i_1j_1} = a_{i_2j_2}, \ a_{k_1j_1} = a_{k_2j_2}, \ a_{i_1l_1} = a_{i_2l_2}$$
are satisfied then $a_{k_1l_1} = a_{k_2l_2}$ \cite[Theorem~1.2.1]{DK15}.  That is, if two ``quadrangles" in a group-based square agree on three points then they agree on the fourth.

\begin{thm}\label{th:infvat}
For every index set~$I$  
there is a Vatican square on~$I$ that is not based on a  group.
\end{thm}

\begin{proof}
Let $\F_\kappa$ be an ordered field on $\kappa=|I|$.
We build a Vatican square $L$ on $\F_\kappa$ with symbol set $\kappa$ by transfinite induction. The strategy to build one on an arbitrary index set is exactly the same as outlined below, potentially starting with a different starting sub-square we call $p_{-1}$ depending on your index set and symbol set.

Again we find it useful to define an appropriate poset, to build the Vatican square by growing a finite one via meeting dense sets. The conditions in our poset $\P$ consist of Vatican regions which are not group based. These are injective partial functions of the form $p:\F_\kappa \times \F_\kappa\To\kappa$. Again we think of these populating $\kappa\times\kappa$ grid, although distances between columns and rows are computed by the ordered field $\F_\kappa$. Moreover, these regions should be Vatican in the sense of a finite square, in that for each $d\in \F_\kappa^+$ such that ${(\F_\kappa)}_{(d)}\neq\emptyset$ we have that each ordered pair of distinct symbols coming from $\kappa$ appear at distance $d$ at most once in each order in rows and at most once in each order in columns. We partially order $\P$ by letting $l \leq m$ so long as $p$ extends $q$ as functions.

We start the induction with the following Vatican region on $\kappa$, call it $l_{-1}$:
$$\begin{array}{cccc}
	\color{blue} 0 & 5 & 6 & \color{blue} 1 \\ 
	7 & \color{blue} 0 & \color{blue}1 & 8  \\
	9 & \color{blue}2 & \color{blue} 3 & 10 \\
	\color{blue} 2 & 11 & 12 & \color{blue} 4 
\end{array}$$
No square containing this as a subsquare can be group-based since it fails to satisfy the quadrangle criterion.
We then begin meeting dense sets in some enumeration of the following four families of dense sets.
We build a descending chain of conditions $l_\alpha$, starting with $l_{-1}$, in $\kappa$ many stages exactly as before. At each stage we use the fact that the union of a descending chain of fewer than $\kappa$ many conditions is itself a condition, making sure that the sequence doesn't grow at too fast a rate (we see below that it is possible to meet each dense set by only adding at most 2 symbols to a condition), and then extend into the dense set under consideration at that stage.

\begin{enumerate}
	\item For each $\alpha, \beta <\kappa$, the set $D_{\alpha,\beta}=\set{l \in \P}{\alpha \text{ appears in row } \beta \text{ of } l}$ is dense.
	
	The trick here is not to mess up the Vatican property on the region. If $\alpha$ does not appear in row $\beta$ of a condition $p$ yet, it is not necessarily enough to simply add it to the end of row $\beta$, since it is possible that then $\alpha$ appears more than once at some distance from another element in that row or column. We are only guaranteed to be safe with that method if $\alpha$ doesn't already appear anywhere in the partial square. Thus the idea is to go far enough out to where it is safe to add $\alpha$ in row $\beta$.
	
	If $l$ doesn't even have a row $\beta$ yet, then on the $\beta$th row, add $\alpha$ at the end, creating a new column with just $\alpha$ in it.
	
	Otherwise, $l$ already has entries in row $\beta$, go far enough out in row $\beta$ (at worst the length of its longest row plus $\beta$ in the field addition of $\F_\kappa$) so that no column or row distance between an element of row $\beta$ of $l$ and $\alpha$ could even have occurred in $p$ initially.
	
	For example, if $\F_\kappa$ has the usual ordering on the natural numbers, in the above square $l_{-1}$, this is how we would add 0 to the bottom row:
$$\begin{array}{cccccccc}
	0 & 5 & 6 & 1 &.&.&.&.\\ 
	7 &  0 & 1 & 8 &.&.&.&. \\
	9 & 2 & 3 & 10&.&.&.&. \\
	2 & 11 & 12 & 4 & . & . & . & \color{red}0
\end{array}$$
and if we would like to add 0 to the sixth row:
$$\begin{array}{ccccccc}
	. & .& .& .& \color{red} 0\\
	. & . &. & . & .\\
	0 & 5 & 6 & 1 & .  \\ 
	7 &  0 & 1 & 8 & . \\
	9 & 2 & 3 & 10 & . \\
	2 & 11 & 12 & 4 & .
\end{array}$$

	\item For each $\alpha, \beta <\kappa$, the set $D^{\alpha,\beta}=\set{l \in \P}{\alpha \text{ appears in column } \beta \text{ of } l}$ is dense.
	
	Same procedure as with rows.	
	
	\item For each $\alpha, \beta < \kappa$, the set $F_{\alpha, \beta}=\set{l \in \P}{\text{there is a symbol in row $\alpha$, column $\beta$ of $l$}}$ is dense.
	
	This dense set guarantees that we fill in the holes left by meeting the other dense sets. If there isn't already an entry in coordinate $(\alpha, \beta)$ of $l$, then pick a symbol that hasn't appeared yet in $p$ and add it to that entry. The condition produced is Latin since that entry only appears once. Since this symbol did not previously appear in any of the entries of $p$, it must also be Vatican.
	
	\item For each $\alpha, \beta<\kappa$ and $d \in \F_\kappa^+$, the set $$E^d_{\alpha, \beta}=\set{l \in \P}{\text{$\alpha$, $\beta$ appear at distance $d$ apart in some row of $l$, in that order}}$$ is dense.
	
	If $\alpha$ and $\beta$ already do not appear distance $d$ apart anywhere in a Vatican region $l$, what we should do to absolutely guarantee we have no conflicts is start a new row. Go far enough out in this row, up to the length of the longest row. Then add $\alpha$, followed by enough symbols so as to be able to add $\beta$ at distance $d$. This indeed will still be a condition, since $\beta$ and $d$ are less than $\kappa$.
	
	For example, with our starting square $l_{-1}$, this is how we would extend it to have 0 and 1 be distance 2 apart (if the field $\F_\kappa$ has the same ordering on finite numbers that the natural numbers do):
$$\begin{array}{ccccccc}
	. &.  &.  & . & \color{red}0 & . &\color{red}1\\
	0 & 5 & 6 & 1 &. &. &. \\ 
	7 &  0 & 1 & 8 &. &. &.  \\
	9 & 2 & 3 & 10 &. &. &. \\
	2 & 11 & 12 & 4 &. &. &. 
\end{array}$$
	
	\item For each $\alpha, \beta<\kappa$ and $d \in \F_\kappa^+$, the set $$E_d^{\alpha, \beta}=\set{l \in \P}{\text{$\alpha$, $\beta$ appear at distance $d$ apart in some column of $p$, in that order}}$$ is dense.
	
	Same procedure as with rows.

\end{enumerate}

We produce a square $L$ at the end of this construction by taking the union of all of the $l_\alpha$'s in the chain we described building above. Clearly $L$ is Latin, since at some stage every ordinal less than $\kappa$ was added to every row and every column as guaranteed by our first three families of dense sets. It must be that $L$ is Vatican as well. First of all, we know that the pair occurrence for each row and column must be satisfied at least once in each row and column by meeting the last two families of dense sets described above. Moreover if this happened somewhere more than once, it would have to happen in some condition $l_\alpha$, but conditions in $\P$ are not allowed to have this property. And $L$ is not
group-based since it includes the square $l_{-1}$.
\end{proof}

\section{Semi-Vatican squares}\label{sec:semivat}

Infinite Semi-Vatican squares---recall that these are squares in which each pair of symbols appears exactly once at each distance~$d$ in rows and columns, rather than exactly once in each order---behave very similarly to Vatican squares.  

The required generalization of~directed $T_{D}$- and~$T_{\infty}$-terraces are directed $S_{D}$- and~$S_{\infty}$-terraces.  As before, let~$I$ be an index set in an ordered field~$\F$.  Let~$G$ be a group of order~$|I|$.  For a bijection~${\bf a}: I \rightarrow G$ define a function~${\bf a}_{(d)}: I \rightarrow G \setminus \{ e \}$ for each~$d \in \F^+$ with~$I_{(d)} \neq \emptyset$ by
$${\bf a}_{(d)}(i) = {\bf a}(i)^{-1}{\bf a}(i+d).$$
If~$G$ has no involutions, and if there is a~$D$ such that for all~$d < D$ with~$I_{(d)} \neq \emptyset$ we have that the image of ${\bf a}_{(d)}$ contains exactly one occurrence from each set~$\{ x,x^{-1} : x \in G\setminus \{e\} \}$, then~${\bf a}$ is a {\em directed $S_D$-terrace for~$G$}.  If ${\bf a}_{(d)}$ has this property for all~$d$ with~$I_{(d)} \neq \emptyset$ then ${\bf a}$ is  a {\em directed $S_{\infty}$-terrace} for~$G$.

The requirement that~$G$ has no involutions comes into play when we consider constructing semi-Vatican squares using the method of Theorem~\ref{th:terrace2square}.  Suppose~$z \in G$ is an involution and~${\bf a}_{(d)} = z$ for some bijection~${\bf a}$ with the usual definition for~${\bf a}_{(d)}$.  If~${\bf a}(i) = x$, then ${\bf a}(i+d) = xz$.  There is a~$j$ such that~${\bf a}(j) = xz$ and now~${\bf a}(i+d) = xz^2 = x$.  Thus the pair~$\{ x, xz \}$ occurs twice at distance~$d$ in~$L({\bf a})$, once in each order.  Hence a square constructed with this method using a group with an involution cannot be semi-Vatican.

The proofs of the previous two sections require only minor modifications to give the following slate of results:

\begin{thm}\label{th:semi_terrace2square}
Let~$G$ be a group of infinite order~$\kappa$ with no involutions.  If~$G$ has a directed $S_{\infty}$-terrace for an index set~$I$ then there is a semi-Vatican square of order~$|G|$.
\end{thm}

\begin{thm}\label{th:semi_T_infty}
Let~$G$ be an involution-free abelian squareful group of infinite order~$\kappa$.   Then~$G$ has a directed $S_{\infty}$-terrace.
\end{thm}

\begin{cor}\label{cor:semi_vatsquares}
For every index set~$I$ there is a semi-Vatican square on~$I$.  In particular, there is a semi-Vatican square of every infinite order.
\end{cor}

\begin{thm}\label{th:semi_notrcls}
For every index set~$I$ 
there is a Latin square on~$I$ that cannot be made semi-Vatican by permuting columns.
\end{thm}

\begin{thm}\label{th:semi_infvat}
For every index set~$I$ there is a semi-Vatican square on~$I$ that is not based on a  group.
\end{thm}

All of the existence results presented so far are non-constructive and rely on transfinite induction.  Perhaps surprisingly, in the case when the group is~$(\R, +)$, the tools of undergraduate calculus are sufficient to construct to a semi-Vatican square.

\begin{thm}\label{th:svr}
There is a semi-Vatican square on index set~$\R$ based on~$(\R,+)$.
\end{thm} 

\noindent
Proof.   We give a direct definition for a directed $S_{\infty}$-terrace~${\bf a}$:
\begin{equation*}
    {\bf a}(x) = \begin{cases}
               e^x   -1            & x \geq 0\\
               -\ln (1-x)       & \text{otherwise}
           \end{cases}
\end{equation*}
This is a continuous, strictly increasing bijection from~$\R$ to~$\R$.  Its derivative is:
\begin{equation*}
    {\bf a}'(x) = \begin{cases}
               e^x               & x \geq 0\\
              \frac{1}{1-x}       & \text{otherwise}
           \end{cases}
\end{equation*}
which is a continuous, strictly increasing bijection from~$\R^+$ to~$\R^+$.

Therefore, for each~$d \in \R^+$, we have that ${\bf a_{(d)}}$ is a bijection from~$\R^+$ to~$\R^+$.  Hence~${\bf a}$ is a directed $S_{\infty}$-terrace and Theorem~\ref{th:semi_terrace2square} gives a semi-Vatican square based on~$I = \R$.
\qed

Similar approaches for Vatican squares quickly run into difficulties.

\section{Orthogonality}\label{sec:orth}

Let~$G$ be a group and~$\theta: G \rightarrow G$ a bijection.  If $g \mapsto g^{-1}\theta(g)$ is a bijection then~$\theta$ is an {\em orthomorphism}; if  $g \mapsto g\theta(g)$ is a bijection then~$\theta$ is a {\em complete mapping}.  Two orthomorphisms, $\theta, \phi$ are {\em orthogonal} if $g \mapsto \theta(g)^{-1} \phi(g)$ is a bijection.

Let~$L$ be the Cayley table of a group~$G = \{ g_i : i \in I \}$ that has $ij^{\rm th}$ entry $g_i g_j$ and for any bijection~$\theta: G \rightarrow G$ define~$L_\theta$ to be the Latin square with $ij^{\rm th}$ entry $g_i \theta(g_j)$.  If~$\theta$ is an orthomorphism then~$L_\theta$ is orthogonal to~$L$ and if~$\theta$ and~$\phi$ are orthogonal orthomorphisms then~$L_\theta$ is orthogonal to~$L_\phi$; see, for example, \cite{Evans07}.

It's known that every infinite group has an orthomorphism~\cite{Bateman50}, so there is a pair of orthogonal Latin squares at every infinite order.  In~\cite{BM91}, countably many orthogonal Latin squares of countable order are constructed using mutually orthogonal orthomorphisms (although without using this terminology) of a specific countably infinite group.

The work of Section~\ref{sec:cayley} may be adapted to give families of mutually orthogonal orthomorphisms.  In the finite case, directed terraces and directed R-terraces are similar objects, with directed terraces giving rise to complete Latin squares and directed R-terraces giving rise to orthogonal Latin squares via orthomorphisms.  Our definitions and results preserve these connections as we move to the infinite.

Let~$I$ be an index set in an ordered field~$\F$.  Let~$G$ be a group of order~$|I|$.  For a bijection~${\bf a}: I \rightarrow G \setminus \{ e \}$ define a function~${\bf a}_{(d)}: I_{(d)} \rightarrow G \setminus \{ e \}$ for each~$d \in \F^+$ with~$I_{(d)} \neq \emptyset$ by
$${\bf a}_{(d)}(i) = {\bf a}(i)^{-1}{\bf a}(i+d).$$ Such a function is called a $R_d$-sequencing for $\a$, if it is a bijection.
If there is a~$D$ such that for all~$d < D$ with~$I_{(d)} \neq \emptyset$ we have that each ${\bf a}_{(d)}$ is a bijection, then~${\bf a}$ is a {\em directed $R_D$-terrace for~$G$}.  If ${\bf a}_{(d)}$ is a bijection for all~$d$  with~$I_{(d)} \neq \emptyset$ then ${\bf a}$ is  a {\em directed $R_{\infty}$-terrace} for~$G$.

\begin{thm}
Let~$G$ be a group of infinite order~$\kappa$.  If~$G$ has a directed $R_{\infty}$-terrace then~$G$ has a set of~$ \kappa$ mutually orthogonal orthomorphisms. 
\end{thm}

\begin{proof}
Let~${\bf a}$ be a directed $R_{\infty}$-terrace for~$G$ over some index set~$I$.  For each~$d$ such that~$I_d \neq \emptyset$ (of which there are~$\kappa$) define~$\theta_d(e) = e$ and $\theta_d( {\bf a}(i) ) = {\bf a}(i+d)$.  This gives us the orthogonal orthomorphisms we're looking for.

First, they are orthomorphisms: given~$g \in G \setminus \{ e \}$ with~${\bf a}(i) = g$ we get
$$g^{-1}\theta_d(g) = {\bf a}(i)^{-1}{\bf a}(i+d) = {\bf a}_{(d)}(i)$$
which, when we also consider that~$\theta_d(e)=e$, gives us a bijection on~$G$.

Second, they are orthogonal: again taking~$g \in G \setminus \{ e \}$ with~${\bf a}(i) = g$, if~$d_2 > d_1$ we get: 
$$\theta_{d_1}^{-1}(g) \theta_{d_2}(g) =  {\bf a}(i+d_1)^{-1}{\bf a}(i+d_2)  = {\bf a}_{(d_2 - d_1)}(i+d_1) .       $$
If~$d_1 < d_2$ we get:
$$\theta_{d_1}^{-1}(g) \theta_{d_2}(g) =  {\bf a}(i+d_1)^{-1}{\bf a}(i+d_2)  = {\bf a}_{(d_1 - d_2)}(i+d_2)^{-1} .       $$
Also~$\theta_{d_1}(e)^{-1}\theta_{d_2}(e) =e$ in each case, giving bijections on~$G$.
\end{proof}

If we have a directed~$R_D$-terrace, then the same argument gives $ | \{ d \leq D : I_{(d)} \neq \emptyset \} |$ mutually orthogonal orthomorphisms for~$G$.

\begin{thm}\label{th:R_infty}
Let~$G$ be an abelian squareful group of infinite order.   Then~$G$ has a directed $R_{\infty}$-terrace.
\end{thm}

\begin{proof}
The only difference between a directed~$R_{\infty}$-terrace and a directed $T_{\infty}$-terrace is that the identity is not in the domain of a directed~$R_{\infty}$-terrace.  The presence of the identity is not relied upon in the Theorem~\ref{th:T_infty}'s proof that abelian squareful groups of infinite order have a  directed $T_{\infty}$ terrace; a simple adjustment of the argument produces the required directed~$R_{\infty}$-terrace.
\end{proof}

This immediately gives:

\begin{cor}
There is a set of~$\kappa$ mutually orthogonal Latin squares of order~$\kappa$ for all infinite cardinalities~$\kappa$.
\end{cor}

As with directed $T_{\infty}$-terraces, we can now ask which infinite groups have (or do not have) directed $R_{\infty}$-terraces.

If we do not insist that our orthomorphisms come from directed R-terraces we can completely remove the restrictions on the groups:

\begin{thm}
Let~$G$ be a group of infinite order~$\kappa$.  Then~$G$ has a set of $\kappa$ mutually orthogonal orthomorphisms and can hence be used to construct a set of~$\kappa$ mutually orthogonal squares of order~$\kappa$.
\end{thm}

\begin{proof}
We build the required set of mutually orthogonal orthomorphisms of~$G$ by transfinite induction on~$\kappa$.  One difference to earlier proofs is that rather than partial functions the approximations to the desired object are sets of partial functions.

Let~$I$ be a set of size~$\kappa$.  Let $\P$ be the poset whose elements are sets of the form~$\d=\{ \theta_i : i \in I \}$ of partial injective functions from~$G$ to~$G$ with $|\bigcup_{i} \dom \theta_i|  < \kappa$ and are sets of mutually orthogonal partial orthomormorphisms in the sense that the partial functions $\eta_i: g \mapsto g^{-1}\theta_i(g)$ and~$\eta_{ij}: g \mapsto \theta_i^{-1}(g)\theta_j(g)$ (where~$i \neq j$) are also injective.  Let~$\a = \{ \theta'_i : i \in I \}$ and~$\b = \{ \theta_i : i \in I \}$.
The relation on~$\P$ is given by~$\a \leq \b$ if and only if for each~$i \in I$ we have that $\theta'_i$ extends~$\theta_i$ as a function;  that is,  $\dom \theta_i \subseteq \dom \theta'_i$ and $\theta'_i \rest \dom \theta_i = \theta_i$.

There are four requirements that the approximations have to meet.  We show that the set of conditions satisfying each of them is dense.

\begin{enumerate}

\item For each~$i \in I$ and~$g \in G$, the set $D_i^g = \{ {\bf d} \in \P : g \in \dom \theta_i \}$ is dense.

Let~${\bf a}$ in~$\P$ with~$g \not\in \dom \theta_i$.  We need a~${\bf d} \in D_i^g$ with~$g \in \dom \theta_i$ and ${\bf d} \leq {\bf a}$.  To extend~${\bf a}$ to~${\bf d}$, choose an $h$ such that $h \not\in \ran \theta_i$, $g^{-1}h \not\in \ran \eta_i$ and $h^{-1}\theta_j(g) \not\in \ran \eta_{ij}$ and set~$\theta_i(g) = h$.  Such an~$h$ exists as the ranges of each of the $\eta_i$ and~$\eta_{ij}$ are smaller than~$\kappa$.

\item For each~$i \in I$ and~$h \in G$, the set $D_i^h = \{ {\bf d} \in \P : h \in \ran \theta_i \}$ is dense.

Let~${\bf a}$ in~$\P$ with~$h \not\in \ran \theta_i$.  We need a~${\bf d} \in D_i^h$ with~$h \in \ran \theta_i$ and ${\bf d} \leq {\bf a}$.  To extend~${\bf a}$ to~${\bf d}$, choose a~$g$ such that $g \not\in \dom \theta_j$ for any~$j$ (including $j=i$) and $g^{-1}h \not\in \ran \eta_i$.  These choices are possible because the restricted sets in each case have size smaller than~$\kappa$.
Set~$\theta_i(g) = h$.

\item For each~$i \in I$ and~$h \in G$, the set $E_i^h = \{ {\bf d} \in \P : h \in \ran \eta_i \}$ is dense.

Let~${\bf a}$ in~$\P$ with~$h \not\in \ran \eta_i$.  We need a~${\bf d} \in E_i^h$ with~$h \in \ran \eta_i$ and ${\bf d} \leq {\bf a}$.  To extend~${\bf a}$ to~${\bf d}$, choose a~$g$ such that $g \not\in \dom \theta_j$ for any~$j$ (including $j=i$) and $gh \not\in \ran \theta_i$.  These choices are possible because the restricted sets in each case have size smaller than~$\kappa$.  Set~$\theta_i(g) = gh$  and hence $\eta_i(g) = h$.

\item For each~$i,j \in I$, with~$i \neq j$, and~$h \in G$, the set $E_{ij}^h = \{ {\bf d} \in \P : h \in \ran \eta_{ij} \}$ is dense.

Let~${\bf a}$ in~$\P$ with~$h \not\in \ran \eta_{ij}$.  We need a~${\bf d} \in E_{ij}^h$ with~$h \in \ran \eta_{ij}$ and ${\bf d} \leq {\bf a}$.  To extend~${\bf a}$ to~${\bf d}$, choose a~$g$ such that $g \not\in \dom \theta_k$ for any~$k$.   Choose a pair~$(h_i,h_j)$ from $(G \setminus \ran \theta_i) \times (G \setminus \ran \theta_j)$ such that~$h_i^{-1}h_j = h$, $g^{-1}h_i \not\in \ran \eta_i$ and $g^{-1}h_j \not\in \eta_j$. These choices are possible because the restricted sets in each case have size smaller than~$\kappa$.
 Set~$\theta_i(g) = h_i$ and~$\theta_j(g) = h_j$.

\end{enumerate}

The transfinite induction now goes through as usual.
\end{proof}

Another embellishment of the complete mapping concept is the strong complete mapping:
If~$\theta$ is both an orthomorphism and a complete mapping then it is a {\em strong complete mapping}.   

In~\cite{Evans12} it is claimed that every countably infinite group has a strong complete mapping.  The following result generalizes that to groups of arbitrary infinite order using essentially the same construction.  However, for the argument to go through we need that the group is squareful (recall that~$G$ is squareful if $|\{g^2 : g \in G \}| = |G|$), a condition that is required for the argument in~\cite{Evans12}, but which was not included.  In Theorem~\ref{thm:InftyAbSCM} we construct strong complete mappings for many non-squareful abelian groups.

\begin{thm}\label{thm:squarefulSCM}
Let~$G$ be a squareful group of infinite order~$\kappa$.  Then~$G$ has a strong complete mapping.
\end{thm}

\begin{proof}
We again use transfinite induction on~$\kappa$.   Let $\P$ be the poset whose elements are partial strong complete mappings (i.e.~injective partial functions~$\theta$ such that $\eta: g \mapsto g^{-1}\theta(g)$ and $\zeta: g \mapsto g\theta(g)$ are also injective) with domain of size less than~$\kappa$.  The relation is given by~$\theta' \leq \theta$ when $\dom \theta \subseteq \dom \theta'$ and $\theta' \rest \dom \theta = \theta$.

We show that the sets of conditions meeting the requirements are dense.

\begin{enumerate}

\item For each~$g \in G$ the set  $D_g = \{ \theta \in \P : g \in \dom \theta \}$ is dense.

Let~$\theta$ in~$\P$ with~$g \not\in \dom \theta$.  We need a~$\theta' \in D_g$ with~$g \in \dom \theta'$ and $\theta' \leq \theta$.  To extend~$\theta$ to~$\theta'$, choose an $h$ such that $h \not\in \ran \theta$, $g^{-1}h \not\in \ran \eta$ and $gh \not\in \ran \zeta$ and set~$\theta'(g) = h$.  Such an~$h$ exists as the ranges of each of the partial functions are smaller than~$\kappa$.

\item For each~$h \in G$ the set  $D_h = \{ \theta \in \P : g \in \ran \theta \}$ is dense.

Let~$\theta$ in~$\P$ with~$h \not\in \ran \theta$.  We need a~$\theta' \in D_h$ with~$h \in \ran \theta'$ and $\theta' \leq \theta$.
To extend~$\theta$ to~$\theta'$,  choose a $g$ such that $g \not\in \dom \theta$, $g^{-1}h \not\in \ran \eta$ and $gh \not\in \ran \zeta$ and set~$\theta'(g) = h$.  Such a~$g$ exists as the number of restricted elements is less than~$\kappa$.

\item For each~$h \in G$ the set  $D_h^{\eta} = \{ \theta \in \P : g \in \ran \eta \}$ is dense.

Let~$\theta$ in~$\P$ with~$h \not\in \ran \eta$.  We need a~$\theta' \in D_h$ with~$h \in \ran \eta'$ and $\theta' \leq \theta$.
To extend~$\theta$ to~$\theta'$,  choose a $g$ such that $g \not\in \dom \theta$, $gh \not\in \ran \theta$ and $g^2h \not\in \ran \zeta$ and set~$\theta'(g) = gh$.  Such a~$g$ exists as~$G$ is squareful and so the number of restricted elements is less than~$\kappa$.

\item For each~$h \in G$ the set  $D_h^{\zeta} = \{ \theta \in \P : g \in \ran \zeta \}$ is dense.

Let~$\theta$ in~$\P$ with~$h \not\in \ran \zeta$.  We need a~$\theta' \in D_h$ with~$h \in \ran \zeta'$ and $\theta' \leq \theta$.
To extend~$\theta$ to~$\theta'$,  choose a $g$ such that $g \not\in \dom \theta$, $g^{-1}h \not\in \ran \theta$ and $g^{-2}h \not\in \ran \eta$ and set~$\theta'(g) = h$.  Such a~$g$ exists as~$G$ is squareful and so the number of restricted elements is less than~$\kappa$.
\end{enumerate}

The result follows.
\end{proof}

Given groups $G_i$, $i\in I$, the {\em direct product}, $\prod_{i\in I}G_i$, is the group with elements $\prod_{i\in I}g_i$, $g_i\in G_i$, multiplication given by $\prod_{i\in I}g_i\prod_{i\in I}h_i=\prod_{i\in I}g_ih_i$. The {\em direct sum}, $\sum_{i\in I}G_i$, is the subgroup of $\prod_{i\in I}G_i$ consisting of elements $\prod_{i\in I}g_i$ in which $g_i\ne 1$ for finitely many values of $i$. 

\begin{thm}\label{thm:direct}
If $G_i$, $i\in I$, are groups that have strong complete mappings, then both $\prod_{i\in I}G_i$ and $\sum_{i\in I}G_i$ have strong complete mappings.
\end{thm}
\begin{proof}
Let $G=\prod_{i\in I}G_i$, and, for each $i\in I$, let $\phi_i$ be a strong complete mapping of $G_i$. It is routine to show that $\phi\colon \prod_{i\in I}g_i\mapsto \prod_{i\in I}\phi_i(g_i)$ is a strong complete mapping of $G$.
 
Let $H=\sum_{i\in I}G_i$, and, for each $i\in I$, let $\phi'_i$ be defined by $\phi'_i(x)=\phi_i(x)\phi_i(1)^{-1}$: $\phi'_i$ is a strong complete mapping of $G_i$ that fixes $1$. It is routine to show that $\phi'\colon \sum_{i\in I}g_i\mapsto \sum_{i\in I}\phi'_i(g_i)$ is a strong complete mapping of $H$.
\end{proof}

Which abelian groups have strong complete mappings? This question is answered for finite abelian groups. 

\begin{thm}\label{thm:finite2gps}
A finite abelian  group has a strong complete mapping if and only if its Sylow $2$-subgroup is trivial or non-cyclic and its Sylow $3$-subgroup is trivial or non-cyclic.
\end{thm}
\begin{proof}
See Theorem~4 in~\cite{Evans12}.
\end{proof}

An important component in the proof of Theorem~\ref{thm:finite2gps} is a quotient group construction of strong complete mappings: such a construction for finite abelian groups is given in~\cite{Evans:1990} and~\cite{Horton:1990}. The construction used works in the infinite case as well.

\begin{thm}\label{thm:G/H}
Let $H$ be a subgroup of an abelian group $G$. If both $H$ and $G/H$ have a strong complete mapping, then so does $G$.
\end{thm}
\begin{proof}
Let $\theta$ be a strong complete mapping of $H$, let $\phi$ be a strong complete mapping of $G/H$, and let $D=\{g_i : i\in I\}$ be a system of distinct coset representatives for $H$ in $G$. Define $\alpha\colon G/H\to D$ by $\alpha(g_i+H)=g_i$, and define $\beta\colon G\to G$ by 
\[
\beta(g_i+h)=\alpha\phi(g_i+H)+\theta(h).
\]
It is routine to show that $\beta$ is a strong complete mapping of $G$.
\end{proof}

It should be noted that no proof of Theorem~\ref{thm:G/H} has been given for non-abelian groups, even in the finite case.

Having answered the question for finite abelian groups, let us turn to infinite abelian groups. The structure of infinite abelian groups is studied in~\cite{Fuchs:2015},~\cite{Kaplansky:1969}, and Chapter 10 of~\cite{Rotman:1995}. Let $A$ be an abelian group written additively and let us define the {\em restricted $2$-subgroup} of $A$ to be the subgroup $E_2$ consisting of the identity and involutions that are not squares. 

\begin{thm}\label{thm:InftyAbSCM}
Let $A$ be an infinite abelian group with restricted $2$-subgroup $E_2$. Then $A$ has a strong complete mapping except, possibly, when $A/E_2$ is finite and has a non-trivial, cyclic Sylow $3$-subgroup.
\end{thm}
\begin{proof}
If $A$ is squareful, then it has a strong complete mapping by Theorem~\ref{thm:squarefulSCM}. Thus, we may assume $A$ not to be squareful, and hence that $E_2$ is infinite. 

We may think of $E_2$ as the additive group of a field $K$. Let $a\in K\setminus\{0,1\}$: $E_2$ has a strong complete mapping, as the mapping $x\mapsto ax$ is a strong complete mapping of $E_2$.

If $A/E_2$ is infinite, then $A/E_2$ is squareful. By Theorem~\ref{thm:G/H}, it follows that, if $A/E_2$ is infinite, then $A$ has a strong complete mapping.

If $A/E_2$ is finite, then, by Theorem~\ref{thm:finite2gps}, $A/E_2$ has a strong complete mapping if the Sylow $2$-subgroup of $A/E_2$ is trivial or non-cyclic, and the Sylow $3$-subgroup of $A/E_2$ is trivial or non-cyclic. Let $S$ be the Sylow $2$-subgroup of $A$. If $S/E_2$ is non-trivial and cyclic, let $g\in S$ be of maximum possible order and let $h\in E_2$. The subgroup $H=\langle g,h\rangle$ has a strong complete mapping by Theorem~\ref{thm:finite2gps}, and $S/H\cong E_2$. Thus $S$ has a strong complete mapping by Theorem~\ref{thm:G/H}. Hence, if the Sylow $3$-subgroup of $A/E_2$ is trivial or non-cyclic, then $A/E_2$ has a strong complete mapping, and hence $A$ has a strong complete mapping.
\end{proof}

In the finite case, strong complete mappings can be used to construct Knut Vic designs and to construct Latin squares orthogonal to both the Cayley table of a group and the normal multiplication table of the group. For a Latin square of order $n$; the {\em $k$th left diagonal} consists of the cells $(i,k+i)$, $i=0,\dots,n-1$, addition modulo $n$; and the {\em $k$th right diagonal} consists of the cells $(i,k-i)$, $i=0,\dots,n-1$, subtraction modulo $n$. A {\em Knut Vic design} of order $n$ is a Latin square of order $n$ in which the entries on each left diagonal, and the entries on each right diagonal form a permutation of the symbol set. In~\cite{Hedayat:1977} and~\cite{Hedayat/Federer:1975} it is shown that a Knut Vic design of order $n$ exists if and only if $\gcd(n,6)=1$. 

For a Latin square $L$ of infinite order~$\kappa$ we may assume that the rows and columns of $L$ are indexed by an abelian group~$A$ of order~$\kappa$.  For~$k \in A$, the {\em $k$th left diagonal} of $L$ consists of the cells $(i,k+i)$, $i\in A$; and the {\em $k$th right diagonal} of $L$ consists of the cells $(i,k-i)$, $i\in A$.  When~$A$ is $\Z$ or~$\R$, these definitions naturally extend the geometric sense of the finite case to the plane.   We say that a Latin square $L$ is a {\em Knut Vic design} if each left diagonal, and each right diagonal contains each symbol exactly once.

\begin{thm}\label{th:infKV}
There exists a Knut Vic design of every infinite order~$\kappa$.
\end{thm}
\begin{proof}
Let~$A$ be a squareful abelian group of order~$\kappa$ and let $\theta$ be a strong complete mapping of $A$: this exists by Theorem~\ref{thm:squarefulSCM}. Let $L$ be the Latin square, with rows and columns indexed by $A$, with symbol set $A$, and with $ij$th entry $i+\theta(j)$. It is routine to show that $L$ is a Knut Vic design.
\end{proof}

Choosing~$A \in \{ \Z, \R \}$ in the proof of Theorem~\ref{th:infKV} gives a Knut Vic design embedded in the plane with a natural geometric interpretation.

Recall that, for a group $G=\{g_i : i\in I\}$, the {\em Cayley table} of $G$ is the Latin square with $ij$th entry $g_ig_j$. The {\em normal multiplication table} of $G$ is the Latin square with $ij$th entry $g_ig_j^{-1}$.

\begin{thm}
If a group $G$ has a strong complete mapping, then there exists a Latin square orthogonal to both the Cayley table of $G$ and the normal multiplication table of $G$.
\end{thm}
\begin{proof}
Let $\theta$ be a strong complete mapping of $G$. Let $L$ be the Latin square with $ij$th entry $g_i\theta(g_j)$. It is routine to show that $L$ is orthogonal to both the Cayley table of $G$ and the normal multiplication table of $G$.
\end{proof}

\section*{Acknowledgements}

This work was partly funded by a Marlboro College Faculty Professional Development Grant and a Marlboro College Town Meeting Scholarship Fund award.  The authors are grateful for this assistance.

\end{document}